\documentclass[11pt]{article}

\usepackage[utf8]{inputenc}
\usepackage[british]{babel}
\usepackage[a4paper, margin=2.5cm]{geometry}
\usepackage{amsmath, amsthm, amssymb, amsfonts}
\usepackage{mathtools}
\usepackage{thmtools}
\usepackage[shortlabels, inline]{enumitem}
\usepackage[font={small, it}]{caption}
\usepackage{microtype}
\usepackage{xcolor}
\usepackage{mathabx}
\usepackage{hyperref}
\usepackage{bm}

\newcommand{\eps}{\varepsilon}
\renewcommand{\phi}{\varphi}
\renewcommand{\rho}{\varrho}

\DeclareMathOperator*{\N}{\mathbb{N}}

\let\setminus=\smallsetminus

\newcommand{\Gnp}{G_{n, p}}

\declaretheorem[parent=section]{theorem}
\declaretheorem[sibling=theorem]{lemma}

\declaretheorem[sibling=theorem]{claim}

\setlist{itemsep=0.1em, topsep=0.1em, parsep=0.1em, partopsep=0.1em}

\hypersetup{
  colorlinks,
  linkcolor={red!60!black},
  citecolor={green!50!black},
  urlcolor={blue!80!black}
}

\colorlet{RoyalRed}{red!70!black}
\definecolor{RoyalBlue}{rgb}{0.25, 0.41, 0.88}
\definecolor{RoyalAzure}{rgb}{0.0, 0.22, 0.66}

\newlength{\bibitemsep}
\newlength{\bibparskip}
\let\oldthebibliography\thebibliography
\renewcommand\thebibliography[1]{%
  \oldthebibliography{#1}%
  \setlength{\parskip}{\bibitemsep}%
  \setlength{\itemsep}{\bibparskip}%
}

\newcounter{propcnt} % counter

\newlist{alphenum}{enumerate}{1}
\setlist[alphenum,1]{%
  label=\normalfont{\bfseries{(\Alph{propcnt}{\arabic*})}},
  ref=\normalfont{(\Alph{propcnt}{\arabic*})},
  leftmargin = \parindent+3em,
}

\title{On the size-Ramsey number of grids}
\author{
  David Conlon\thanks{Department of Mathematics, California Institute of
  Technology, Pasadena, CA 91125, USA. Email: \texttt{dconlon@caltech.edu}.
  Research supported by NSF Award DMS-2054452.}
  \and
  Rajko Nenadov\thanks{Google Z\"urich. Email: \texttt{rajkon@gmail.com}.}
  \and
  Milo\v{s} Truji\'{c}\thanks{Institute of Theoretical Computer Science, ETH
  Z\"{u}rich, 8092 Z\"{u}rich, Switzerland. Email: \texttt{mtrujic@inf.ethz.ch}.
  Research supported by grant no.\ 200020 197138 of the Swiss National Science
  Foundation.}
}
\date{}

\begin{document}
\maketitle

\begin{abstract}
  We show that the size-Ramsey number of the $\sqrt{n} \times \sqrt{n}$ grid
  graph is $O(n^{5/4})$, improving a previous bound of $n^{3/2 + o(1)}$ by
  Clemens, Miralaei, Reding, Schacht, and Taraz.
\end{abstract}

\section{Introduction}

For graphs $G$ and $H$, we say that $G$ is \emph{Ramsey} for $H$, and write $G
\rightarrow H$, if every $2$-colouring of the edges of $G$ contains a
monochromatic copy of $H$. In 1978, Erd\H{o}s, Faudree, Rousseau, and
Schelp~\cite{erdHos1978size} pioneered the study of the \emph{size-Ramsey
number} $\hat r(H)$, defined as the smallest integer $m$ for which there exists
a graph $G$ with $m$ edges such that $G \rightarrow H$. The existence of the
usual Ramsey number $r(H)$ shows that this notion is sensible, since, for any
$H$, it is easy to see that $\hat r(H) \leq \binom{r(H)}{2}$. When $H$ is a
complete graph, this inequality is an equality, a simple fact first observed by
Chv\'{a}tal.

An early example showing that size-Ramsey numbers can exhibit interesting
behaviour was found by Beck~\cite{beck1983size}, who showed that $P_n$, the path
with $n$ vertices, satisfies $\hat r(P_n) = O(n)$, which is significantly
smaller than the $O(n^2)$ bound that follows from applying the inequality above
and the corresponding bound $r(P_n) = O(n)$ for the usual Ramsey number of
$P_n$. In a follow-up paper, Beck~\cite{beck1990size} asked whether a similar
phenomenon occurs for all bounded-degree graphs, that is, whether, for any
integer $\Delta \geq 3$, there exists a constant $c$ such that any graph $H$
with $n$ vertices and maximum degree $\Delta$ has size-Ramsey number at most
$cn$. Although R\"{o}dl and Szemer\'{e}di~\cite{rodl2000size} showed that this
question has a negative answer already for $\Delta=3$, much work has gone into
extending Beck’s result to other natural families of graphs, including:
cycles~\cite{haxell1995induced}, bounded-degree
trees~\cite{friedman1987expanding}, powers of paths and bounded-degree
trees~\cite{berger2020size, clemens2019size, han2020multicolour}, and more
besides.

Most of the known families with linear size-Ramsey numbers have a bounded
structural parameter, such as bandwidth~\cite{clemens2019size} or, more
generally, treewidth~\cite{kamcev2021size} (though see the recent
papers~\cite{draganic2022rolling, letzter2021size} for examples with a somewhat
different flavour). However, a fairly simple family of graphs which does not
fall into any of these categories, but may still have linear size-Ramsey
numbers, is the family of \emph{two-dimensional grid graphs}. For $s \in \N$,
the $s \times s$ grid is the graph with vertex set $[s] \times [s]$ where two
pairs are adjacent if and only if they differ by one in exactly one coordinate.
Obviously, the maximum degree of the $s \times s$ grid is four, but its
bandwidth and treewidth are both exactly $s$ (see, e.g.,
\cite{chvatalova1975optimal}), so the problem of estimating the size-Ramsey
number of this graph, and usually we will take $s = \sqrt{n}$ so that the graph
has $n$ vertices, provides an interesting test case for exploring new ideas and
techniques.

Regarding upper bounds for the size-Ramsey number of the $\sqrt{n} \times
\sqrt{n}$ grid, an important result of Kohayakawa, R\"{o}dl, Schacht, and
Szemer\'{e}di~\cite{kohayakawa2011sparse}, which says that every graph $H$ with
$n$ vertices and maximum degree $\Delta$ satisfies $\hat r(H) \leq
n^{2-1/\Delta+o(1)}$, immediately yields the bound $n^{7/4+o(1)}$. This was
recently improved by Clemens, Miralaei, Reding, Schacht, and
Taraz~\cite{clemens2021size} to $n^{3/2+o(1)}$ (and an alternative proof of this
bound was also noted in our recent paper~\cite{conlon2022size}). The goal of
this short note is to provide an elementary proof of an improved upper bound.

\begin{theorem}\label{thm:main-theorem}
  There exists a constant $C > 0$ such that the size-Ramsey number of the
  $\sqrt{n} \times \sqrt{n}$ grid graph is at most $Cn^{5/4}$.
\end{theorem}

Like much of the work on size-Ramsey numbers, the previous bounds for grids were
obtained by applying the sparse regularity method to show that every
$2$-colouring of the edges of the Erd\H{o}s--R\'{e}nyi random graph $\Gnp$, for
some appropriate density $p$, contains a monochromatic copy of the grid.
However, it is a simple exercise in the first moment method to show that for $p
\ll n^{-1/2}$ the random graph $\Gnp$ with high probability does not contain the
$s \times s$ grid graph as a subgraph if $s = \Theta(\sqrt{n})$, so the bound
$O(n^{3/2})$ is the best that one can hope to achieve using this procedure.

To see how it is that we gain on this bound, suppose that $s = \sqrt{n}$. It is
known~\cite{haxell1995induced} that there are $K, \Delta > 0$ and a graph $H$
with $Ks$ vertices and maximum degree at most $\Delta$ which is Ramsey for
$C_s$, the cycle of length $s$. Consider now a `blow-up' $\Gamma$ of $H$
obtained by replacing every $x \in V(H)$ by an independent set $V_x$ of order
$\Theta(s)$ and every $xy \in H$ by a bipartite graph $(V_x,V_y)$ in which every
edge exists independently with probability $p = \Theta(s^{-1/2})$. With high
probability, such a blow-up contains $\Theta(s^{5/2}) = \Theta(n^{5/4})$ edges.
That is, instead of revealing a random graph $\Gnp$ on all $n = \Theta(s^2)$
vertices, we only reveal edges that lie within $\Theta(s)$ bipartite subgraphs,
each with parts of order $\Theta(s)$. This salvages a significant number of
edges which would otherwise go to waste.

Consider now a $2$-colouring of $\Gamma$ and recall that $H$ was chosen so that
$H \rightarrow C_s$. A key lemma, Lemma~\ref{lem:regular-subgraph} below, then
allows us to conclude that there are sets $V_1,\dotsc,V_s$ in $\Gamma$ and a
collection $U_i \subseteq V_i$ of large subsets such that all $(U_i,U_{i+1})$
with $i \in [s]$, where addition is taken modulo $s$, are `regular' in the same
colour. We may then sequentially embed the vertices of the grid so that the
first row is embedded into $U_1,\dotsc,U_s$, the second into
$U_2,\dotsc,U_s,U_1$, and so on.

\section{Definitions and key lemmas}

In this section, we recall several standard definitions and note two key lemmas
that will be needed in the proof of Theorem~\ref{thm:main-theorem}. Most of
these revolve around the concept of \emph{sparse regularity} (for a thorough
overview of which we refer the reader to the survey by Gerke and
Steger~\cite{gerke2005sparse}).

For $\eps > 0$ and $p \in (0, 1]$, a pair of sets $(V_1, V_2)$ is said to be
$(\eps, p)$-\emph{lower-regular} in a graph $G$ if, for all $U_i \subseteq V_i$,
$i \in \{1,2\}$, with $|U_i| \geq \eps|V_i|$, the density $d_G(U_1, U_2) =
e_G(U_1, U_2)/(|U_1||U_2|)$ of edges between $U_1$ and $U_2$ satisfies
\[
  d_G(U_1, U_2) \geq (1-\eps)p.
\]
Immediately from this definition, we get that in every $(\eps,p)$-lower-regular
pair $(V_1, V_2)$, for each $i \in \{1,2\}$, all but at most $\eps|V_i|$
vertices in $V_i$ have degree at least $(1-\eps)p|V_{3-i}|$ into $V_{3-i}$ --- a
fact we will make use of in the proof of Theorem~\ref{thm:main-theorem}. Another
useful and well-known property is that lower-regularity is inherited on large
sets.

\begin{lemma}
  \label{lem:slicing-lemma}
  Let $0 < \eps < \delta$, $p \in (0, 1]$, and let $(V_1, V_2)$ be an $(\eps,
  p)$-lower-regular pair. Then any pair of subsets $V_i' \subseteq V_i$, $i \in
  \{1,2\}$, with $|V_i'| \geq \delta|V_i|$ form an
  $(\eps/\delta,p)$-lower-regular pair.
\end{lemma}

For $\lambda > 0$ and $p \in (0, 1]$, a graph $G$ is said to be
$(\lambda,p)$-\emph{uniform} if, for all disjoint $X, Y \subseteq V(G)$ with
$|X|, |Y| \geq \lambda|V(G)|$, the density of edges between $X$ and $Y$ satisfies
$d_G(X,Y) = (1\pm\lambda)p$. If only the upper bound holds, the graph is said to
be \emph{upper-uniform}.\footnote{For consistency with the existing literature
and for historical reasons, we use both `regular' and `uniform' as terms, even
though they are basically the same concept.} For example, it is easy to see that
the random graph $\Gnp$ is with high probability $(o(1),p)$-uniform whenever $p
\gg 1/n$. If $G = (V_1, V_2; E)$ is bipartite, we say that $G$ is $(\lambda,
p)$-uniform or upper-uniform if the same conditions hold for all $X \subseteq
V_1$ and $Y \subseteq V_2$ with $|X| \geq \lambda |V_1|$ and $|Y| \geq \lambda
|V_2|$. In order to prove our main technical lemma, we rely on the following
result, a simple corollary of \cite[Lemma~6]{kohayakawa2018anti}, whose proof
follows a density increment argument. The same conclusion can also be obtained
by an application of the sparse regularity lemma.

\begin{lemma}\label{lem:regular-edge}
  For all $0 < \eps < 1/2$ and $\alpha \in (0, 1)$, there exists $\lambda > 0$
  such that the following holds for every $p \in (0, 1]$. Let $G = (V_1, V_2;
  E)$ be a $(\lambda, p)$-upper-uniform bipartite graph with $|V_1| = |V_2|$ and
  $|E| \geq \alpha|V_1||V_2|p$. Then there exist $U_i \subseteq V_i$, $i \in
  \{1,2\}$, with $|U_i| = \lambda|V_i|$ such that $(U_1, U_2)$ is $(\eps, \alpha
  p)$-lower-regular in $G$.
\end{lemma}

The next lemma is the crux of our argument. Here and elsewhere, we say that
$(X,Y; E)$ is \emph{lower-regular} if $(X,Y)$ is lower-regular with respect to
the set of edges $E$.

\begin{lemma}\label{lem:regular-subgraph}
  For every $r, \Delta \geq 2$ and $\eps > 0$, there exists $\lambda > 0$ such
  that the following holds for every $p \in (0, 1]$. Let $H$ be a graph on at
  least two vertices with $\Delta(H) \leq \Delta$ and let $\Gamma$ be obtained
  by replacing every $x \in V(H)$ with an independent set $V_x$ of sufficiently
  large order $n$ and every $xy \in H$ by a $(\lambda,p)$-uniform bipartite
  graph between $V_x$ and $V_y$. Then, for every $r$-colouring of the edges of
  $\Gamma$, there exists an $r$-colouring $\phi$ of the edges of $H$ and, for
  every $x \in V(H)$, a subset $U_x \subseteq V_x$ of order $|U_x| = \lambda n$
  such that $(U_x, U_y; E_{\phi(xy)})$ is $(\eps,p/(2r))$-lower-regular for each
  $xy \in H$, where $E_{\phi(xy)} \subseteq E(\Gamma)$ stands for the edges in
  colour $\phi(xy)$.
\end{lemma}

\begin{proof}
  Given $\eps$, $r$, and $\Delta$, we let $\alpha = 1/(2r)$, $\eps_{\Delta+1} :=
  \eps$, $\lambda_{\Delta+1} =
  \lambda_{\ref{lem:regular-edge}}(\eps_{\Delta+1},\alpha)$, and, for every $i =
  \Delta, \dotsc, 1$, sequentially take $\eps_i = \eps_{i+1}\lambda_{i+1}$ and
  $\lambda_i = \lambda_{\ref{lem:regular-edge}}(\eps_i,\alpha)$. Lastly, let
  $\lambda = \prod_{i \in [\Delta+1]} \lambda_i$.

  Fix any $r$-colouring of (the edges of) $\Gamma$ and, for every $c \in [r]$,
  let $\Gamma_c$ stand for the subgraph (in terms of edges) in colour $c$. Note
  that $H$ has edge-chromatic number at most $\Delta+1$. In other words, there
  exists a partition of the edges of $H$ into $H_1,\dotsc,H_{\Delta+1}$ such
  that each $H_i$ is a matching. We find the required collection $\{U_x\}_{x \in
  V(H)}$ by maintaining the following condition for every $i \in [\Delta+1]$:
  for every $x \in V(H)$, there exists a chain $V_x = U_x^0 \supseteq U_x^1
  \supseteq U_x^2 \supseteq \dotsb \supseteq U_x^i$ such that
  \begin{enumerate}[label=(\emph{\roman*}), ref=(\emph{\roman*})]
    \item\label{reg-sub-size} $|U_x^j| = \lambda_j|U_x^{j-1}|$ for all $j \in
      [i]$ and
    \item\label{reg-sub-low-reg} for every $xy \in \bigcup_{j \leq i} H_j$,
      $(U_x^i, U_y^i)$ is $(\eps_i,\alpha p)$-lower-regular in $\Gamma_c$ for
      some $c \in [r]$.
  \end{enumerate}

  Consequently, for $i = \Delta+1$, we obtain sets $U_x \subseteq V_x$, for
  every $x \in V(H)$, of order $|U_x| = (\prod_{i \in [\Delta+1]} \lambda_i) n =
  \lambda n$ such that $(U_x,U_y)$ is $(\eps_{\Delta+1},\alpha p)$-lower-regular
  and, thus, $(\eps,\alpha p)$-lower-regular for every $xy \in H$. It remains to
  show that we can indeed do this.

  Consider first $i = 1$. For each $xy \in H_1$, let $c \in [r]$ be the majority
  colour in $\Gamma[V_x,V_y]$. As $e_{\Gamma_c}(V_x, V_y) \geq
  (1-\lambda)n^2p/r$, we may apply Lemma~\ref{lem:regular-edge} with $\eps_1$
  (as $\eps$) and $\Gamma_c[V_x,V_y]$ (as $G$) to obtain sets $U_x^1, U_y^1$
  with the desired properties. For every $x \in V(H)$ which is isolated in
  $H_1$, we simply take an arbitrary subset $U_x^1 \subseteq V_x$ of order
  $\lambda_1|U_x^0|$. Thus, the required condition holds for $i=1$.

  Suppose now that the condition holds for some $i \geq 1$ and let us show that
  it also holds for $i+1$. As above, for every $xy \in H_{i+1}$, let $c \in [r]$
  be the majority colour in $\Gamma[V_x,V_y]$. Since $\Gamma[V_x, V_y]$ is
  $(\lambda,p)$-uniform and, by \ref{reg-sub-size}, $|U_x^i|, |U_y^i| \geq
  \lambda n$, we have $e_\Gamma(U_x^i, U_y^i) = (1 \pm \lambda)|U_x^i||U_y^i|p$
  and, hence,
  \[
    (1-\lambda)|U_x^i||U_y^i|p/r \leq e_{\Gamma_c}(U_x^i,U_y^i) \leq
    (1+\lambda)|U_x^i||U_y^i|p.
  \]
  Lemma~\ref{lem:regular-edge} applied to $\Gamma_c[U_x^i,U_y^i]$ with
  $\eps_{i+1}$ (as $\eps$) gives sets $U_x^{i+1} \subseteq U_x^i$ and $U_y^{i+1}
  \subseteq U_y^i$ of order
  \[
    |U_x^{i+1}| = \lambda_{i+1}|U_x^i| \qquad \text{and} \qquad |U_y^{i+1}| =
    \lambda_{i+1}|U_y^i|
  \]
  for which $(U_x^{i+1},U_y^{i+1})$ is $(\eps_{i+1},\alpha p)$-lower-regular in
  $\Gamma_c$. For every $x \in V(H)$ which is isolated in $H_{i+1}$, we again
  take an arbitrary subset $U_x^{i+1} \subseteq U_x^i$ of order
  $\lambda_{i+1}|U_x^i|$. Observe also that, for every $xz \in \bigcup_{j \leq
  i} H_j$, since $(U_x^i,U_z^i)$ was $(\eps_i,\alpha p)$-lower-regular in
  $\Gamma_{c'}$ for some $c' \in [r]$ and $|U_x^{i+1}| = \lambda_{i+1}|U_x^i|$,
  Lemma~\ref{lem:slicing-lemma} and the fact that $\eps_i/\lambda_{i+1} =
  \eps_{i+1}$ imply that $(U_x^{i+1},U_z^{i+1})$ is $(\eps_{i+1},\alpha
  p)$-lower-regular in $\Gamma_{c'}$, as desired. This completes the proof.
\end{proof}

We also need a variant of a result from our previous
paper~\cite[Lemma~3.5]{conlon2022size} about regularity inheritance. While that
result was stated for the usual (full) notion of regularity, we only need
lower-regularity here, allowing us to save a factor of $(\log n)^{1/2}$.

\begin{lemma}\label{lem:typical-vertices}
  For all $\eps, \alpha, \lambda > 0$, there exist positive constants
  $\eps'(\eps,\alpha)$ and $C(\eps,\alpha,\lambda)$ such that for $p \geq
  Cn^{-1/2}$, with probability at least $1 - o(n^{-5})$, the random graph
  $\Gamma \sim \Gnp$ has the following property.

  Suppose $G \subseteq \Gamma$ and $V_1, V_2 \subseteq V(\Gamma)$ are disjoint
  subsets of order $\tilde n = \lambda n$ such that $(V_1, V_2)$ is
  $(\eps',\alpha p)$-lower-regular in $G$. Then there exists $B \subseteq
  V(\Gamma)$ of order $|B| \leq \eps\tilde n$ such that, for each $v, w \in
  V(\Gamma) \setminus (V_1 \cup V_2 \cup B)$ (not necessarily distinct), the
  following holds: for any two subsets $N_v \subseteq N_\Gamma(v, V_1)$ and $N_w
  \subseteq N_\Gamma(w, V_2)$ of order $\alpha \tilde n p/4$, both $(N_v, V_2)$
  and $(N_v, N_w)$ are $(\eps,\alpha p)$-lower-regular in $G$.
\end{lemma}

\begin{proof}[Sketch of the proof.]
  The proof proceeds along the same lines as the proof of
  \cite[Lemma~3.5]{conlon2022size}. The only difference is that there we made
  use of an inheritance lemma for \emph{full regularity} (namely, Corollary~3.5
  in \cite{vskoric2018local}), which requires the sets on which regularity is
  inherited to be of order at least $C\log n/p$, resulting in the requirement
  that $p \geq C(\log n/n)^{1/2}$. However, for lower-regularity, one can
  instead use the inheritance lemma of Gerke, Kohayakawa, R\"{o}dl, and
  Steger~\cite[Corollary~3.8]{gerke2007small}, which only requires the sets to
  be of order at least $C/p$, resulting in $p \geq Cn^{-1/2}$. The rest of the
  proof remains exactly the same.
\end{proof}

\section{Proof of Theorem~\ref{thm:main-theorem}}

Since it requires no additional work, we will actually prove the $r$-colour
analogue of Theorem~\ref{thm:main-theorem}. More precisely, we will show that
for every integer $r \geq 2$ there exists a graph of order $n$ with $O(n^{5/4}$)
edges for which every $r$-colouring of the edges contains a monochromatic copy
of the $\delta\sqrt{n} \times \delta\sqrt{n}$ grid for some $\delta > 0$.

By a result of Haxell, Kohayakawa, and
{\L}uczak~\cite[Theorem~10]{haxell1995induced}, there exist constants $K, \Delta
> 0$, both depending only on $r$, such that, for every sufficiently large $s \in
\N$, there is a graph $H$ on $Ks$ vertices with maximum degree at most $\Delta$
which has the property that every $r$-colouring of its edges contains a
monochromatic copy of $C_\ell$, the cycle of length $\ell$, for every $\log s
\ll \ell \leq s$. Let
\[
  \alpha = 1/(2r), \enskip \eps = \alpha/256, \enskip \eps' =
  \eps'_{\ref{lem:typical-vertices}}(\eps/9, \alpha), \enskip \lambda =
  \lambda_{\ref{lem:regular-subgraph}}(r, \Delta, \eps'), \enskip \text{and}
  \enskip \delta = \min\{1/(4K), \eps\lambda/4 \}.
\]
We show that the size-Ramsey number of the $\delta s \times \delta s$ grid is
$O(s^{5/2})$, which, for $s = \sqrt{n}$, implies the desired statement.

Let $\Gamma$ be a graph obtained by replacing every vertex $x \in V(H)$ by an
independent set $V_x$ of order $s$ and every edge $xy \in H$ by a bipartite
graph between $V_x$ and $V_y$ in which each edge exists independently with
probability $p = Cs^{-1/2}$ for some sufficiently large constant $C > 0$. With
high probability, $\Gamma$ has the following property:
\stepcounter{propcnt}
\begin{alphenum}
  \item\label{prp:upper-uniform} $e_\Gamma(V_x', V_y') =
    (1\pm\lambda)|V_x'||V_y'| p$ for every $xy \in H$ and $V_x' \subseteq V_x$
    and $V_y' \subseteq V_y$ with $|V_x'||V_y'|p \geq 100s/\lambda^2$.
\end{alphenum}
This is a standard feature of random graphs and follows from the Chernoff bound
together with an application of the union bound. In particular, it establishes
that with high probability $\Gamma[V_x,V_y]$ is $(\lambda,p)$-uniform for every
$xy \in H$ and, therefore, $\Gamma$ has at most
\[
  Ks \cdot \Delta/2 \cdot (1+\lambda)s^2p = O(s^{5/2})
\]
edges. Additionally, with high probability, $\Gamma$ is such that every
$\Gamma[V_x \cup V_y \cup V_z]$ has the property of
Lemma~\ref{lem:typical-vertices} (applied with $\eps/9$ as $\eps$, $\lambda/3$
as $\lambda$, and $3s$ as $n$) for every path $xyz$ of length two in
$H$.\footnote{Technically, to apply the lemma, we must also temporarily reveal
the edges between $V_x$ and $V_z$ and within each $V_x, V_y, V_z$, but, unless
$xz$ is itself an edge of $H$, these are all then removed from $\Gamma$.} This
again follows from the union bound, as there are $O(s)$ such paths in total and
the conclusion of Lemma~\ref{lem:typical-vertices} holds with probability $1 -
o(s^{-5})$ for every fixed path. We now fix an outcome of $\Gamma$ which
satisfies all of these properties.

Consider some $r$-colouring of the edges of $\Gamma$ and let $\phi$ be the
colouring of the edges of $H$ given by Lemma~\ref{lem:regular-subgraph} (applied
with $\varepsilon'$ as $\varepsilon$). By the choice of $H$, this colouring
contains a monochromatic copy of $C_{\delta s}$, which, without loss of
generality, we may assume has vertices $1,\dotsc,\delta s$. Therefore, there is
a colour $c \in [r]$ and sets $U_i$ of order $\tilde s = \lambda s$ in $\Gamma$
such that, for every $i \in [\delta s]$, the pair $(U_i,U_{i+1})$ is
$(\eps',\alpha p)$-lower-regular in the subgraph of $\Gamma$ induced by colour
$c$, where we identify $\delta s + i$ with $i$. Let $G$ be the graph induced by
these sets whose edges are the edges of $\Gamma$ of colour $c$. We will show
that $G$ contains the $\delta s \times \delta s$ grid as a subgraph.

For every $i \in [\delta s]$, let $B \subseteq U_i \cup U_{i+1} \cup U_{i+2}$ be
the set given by Lemma~\ref{lem:typical-vertices} (which was applied with
$\eps/9$ as $\eps$, $\lambda/3$ as $\lambda$, and $3s$ as $n$) on $\Gamma[U_i
\cup U_{i+1} \cup U_{i+2}]$, which is a set of `bad vertices' for the pair
$(U_{i+1}, U_{i+2})$. As each $U_i$ is a part of three such applications, by
the chosen properties of $\Gamma$, for every $i \in [\delta s]$ there exists a
set $B_i \subseteq U_i$ of order $|B_i| \leq \eps\tilde s$ such that:
\stepcounter{propcnt}
\begin{alphenum}
  \item\label{prp:one-sided-inh} $(N_v, U_{i+2} \setminus B_{i+2})$ is $(\eps,
    \alpha p)$-lower-regular\footnote{The conclusion of
    Lemma~\ref{lem:typical-vertices} states that $(N_v, U_{i+2})$ is
    $(\eps/9,\alpha p)$-lower-regular, but, as $B_{i+2}$ is small,
    Lemma~\ref{lem:slicing-lemma} implies that $(N_v, U_{i+2} \setminus
    B_{i+2})$ is $(\eps,\alpha p)$-lower-regular.} in $G$ for every $v \in U_i
    \setminus B_i$ and $N_v \subseteq N_G(v, U_{i+1})$ of order $\alpha\tilde
    sp/4$ and
  \item\label{prp:two-sided-inh} $(N_v, N_u)$ is $(\eps, \alpha
    p)$-lower-regular in $G$ for every $v \in U_i \setminus B_i$, $u \in U_{i+1}
    \setminus B_{i+1}$ and $N_v \subseteq N_G(v, U_{i+1})$, $N_u \subseteq
    N_G(u, U_{i+2})$, each of order $\alpha\tilde sp/4$.
\end{alphenum}

Our plan is to embed the vertex $(i, j)$ of the $\delta s \times \delta s$ grid
into $U_{i+j-1}$. The next claim helps us achieve this.

\begin{claim}\label{cl:embedding-invariant}
  Let $i \in [\delta s]$. Suppose that sets $S_{i+j-1} \subseteq U_{i+j-1}
  \setminus B_{i+j-1}$ of order $\alpha\tilde sp/4$ are given for each $j \in
  [\delta s]$ and that $(S_{i+j-1}, S_{i+j})$ and $(S_{i+j-1}, U_{i+j} \setminus
  B_{i+j})$ are $(\eps, \alpha p)$-lower-regular. Then, for every $Q_{i+j-1}
  \subseteq U_{i+j-1}$, $j \in [\delta s]$, of order $|Q_{i+j-1}| \leq 2\eps
  \tilde s$, there exists a path $v_1,\dotsc,v_{\delta s}$ with each $v_j \in
  S_{i+j-1}$ such that $|N_G(v_j, U_{i+j} \setminus Q_{i+j})| \geq \alpha\tilde
  sp/4$.
\end{claim}

\begin{figure}[!htbp]
  \centering
  \includegraphics[width=\textwidth]{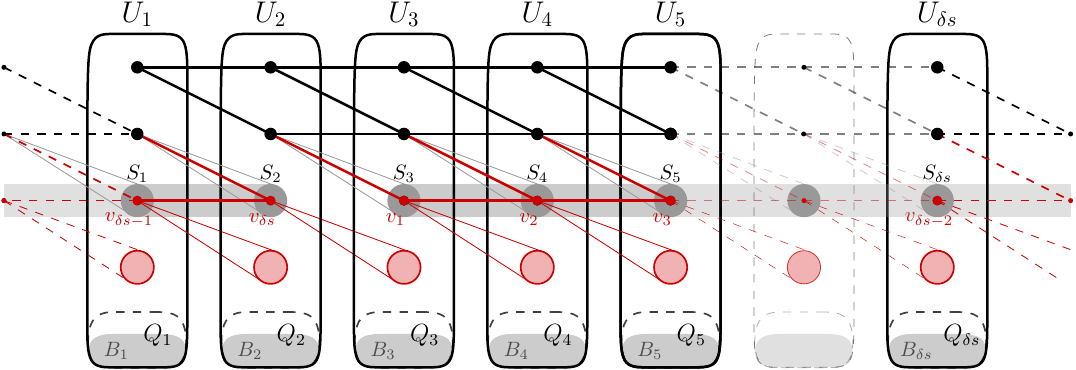}
  \caption{A picture showing the first two rows of the grid already embedded
    (the thick black lines), the candidate sets for the third row (the grey
    blobs $S_3, S_4, \dotsc, S_{\delta s}, S_1, S_2$), and (in red) the path
    $v_1, v_2, \dotsc, v_{\delta s}$ given by
    Claim~\ref{cl:embedding-invariant}, together with the corresponding
    neighbourhoods $N_G(v_j, U_{i+j} \setminus Q_{i+j})$ (the red blobs).}
  \label{fig:embedding}
\end{figure}

Before proving the claim, we show how to complete the embedding of the grid
assuming that it holds. We start by embedding the first row. Let $v_1 \in U_1
\setminus B_1$ be a vertex for which there is $S_2 \subseteq N_G(v_1, U_2
\setminus B_2)$ of order $\alpha\tilde sp/4$ such that $(S_2, U_3 \setminus
B_3)$ is $(\eps, \alpha p)$-lower-regular. As $(U_1 \setminus B_1, U_2 \setminus
B_2)$ is $(2\eps', \alpha p)$-lower-regular, there are at least
$(1-2\eps')(1-\eps)\tilde s$ vertices $v \in U_1 \setminus B_1$ that satisfy
\[
  \deg_G(v, U_2 \setminus B_2) \geq (1-2\eps')|U_2 \setminus B_2|\alpha p \geq
  \alpha\tilde sp/4,
\]
by our choice of constants. Thus, by property \ref{prp:one-sided-inh} almost any
choice of $v_1 \in U_1 \setminus B_1$ will do. Sequentially, for every $i \geq
2$, let $v_i \in S_i$ be a vertex for which there is $S_{i+1} \subseteq N_G(v_i,
U_{i+1} \setminus B_{i+1})$ of order $\alpha\tilde sp/4$ and both $(S_{i+1},
U_{i+2} \setminus B_{i+2})$ and $(S_i, S_{i+1})$ are $(\eps, \alpha
p)$-lower-regular. This is possible as $(S_i, U_{i+1} \setminus B_{i+1})$ is
$(\eps, \alpha p)$-lower-regular and properties \ref{prp:one-sided-inh} and
\ref{prp:two-sided-inh} hold. We continue until we have embedded the first row
of the grid as $v_1,\dotsc,v_{\delta s}$, with $v_i \in U_i$ for every $i \in
[\delta s]$.

Consider now sets $S_2, \dotsc, S_{\delta s}, S_1$ which we previously chose,
where we note that $S_1$ was defined when we embedded $v_{\delta s}$. In
particular, $S_{1+j} \subseteq U_{1+j} \setminus B_{1+j}$ and $(S_{1+j},
S_{2+j})$ and $(S_{1+j}, U_{2+j} \setminus B_{2+j})$ are both $(\eps, \alpha
p)$-regular for every $j \in [\delta s]$. Then, by setting $Q_{1+j} := B_{1+j}
\cup \{v_{1+j}\}$ and invoking Claim~\ref{cl:embedding-invariant} with $i = 2$,
we can embed the second row of the grid as $u_1, \dotsc, u_{\delta s}$, with
$u_j \in S_{1+j}$ for every $j \in [\delta s]$. By the conclusion of
Claim~\ref{cl:embedding-invariant} and a slight abuse of notation, there is a
collection of sets $S_{2+j} \subseteq N_G(u_j, U_{2+j} \setminus Q_{2+j})$ for
every $j \in [\delta s]$, each of order $\alpha\tilde sp/4$, which, by
\ref{prp:one-sided-inh} and \ref{prp:two-sided-inh}, as $u_j \in U_{1+j}
\setminus B_{1+j}$ and $u_{j+1} \in U_{1+j+1} \setminus B_{1+j+1}$, are such
that $(S_{2+j}, S_{2+j+1})$ and $(S_{2+j}, U_{2+j+1} \setminus B_{2+j+1})$ are
$(\eps, \alpha p)$-lower-regular.

The same process can now be repeated for any $i \geq 3$ by setting the sets
$Q_{i+j-1} \subseteq U_{i+j-1}$ for every $j \in [\delta s]$ to be the union of
$B_{i+j-1}$ and the vertices of the grid that were previously embedded into
$U_{i+j-1}$, that is, the images of the vertices $(1,i+j-1), (2,i+j-2), \dotsc,
(i-1,j+1)$. Since $|B_{i+j-1}| \leq \eps\tilde s$, $\delta < \eps\lambda$, and
the lower-regularity conditions hold by \ref{prp:one-sided-inh} and
\ref{prp:two-sided-inh}, we may apply Claim~\ref{cl:embedding-invariant} to
embed the $i$th row. It only remains to prove this claim.

\begin{proof}[Proof of Claim~\ref{cl:embedding-invariant}]
  Without loss of generality, we may assume that all the $Q_{i+j-1}$ are of
  order $2\eps\tilde s$, as we can take arbitrary supersets if this is not the
  case. Let $S_{i+j-1}' \subseteq S_{i+j-1}$ be the set of all $v \in S_{i+j-1}$
  with at least $\alpha\tilde sp/4$ neighbours in $U_{i+j} \setminus Q_{i+j}$.
  On the one hand, as $(S_{i+j-1}, U_{i+j} \setminus B_{i+j})$ is $(\eps,\alpha
  p)$-lower-regular and, thus, there are fewer than $\eps|S_{i+j-1}|$ vertices
  in $S_{i+j-1}$ with degree less than $\alpha\tilde sp/2$ in $U_{i+j} \setminus
  B_{i+j}$, we have
  \[
    e_G(S_{i+j-1} \setminus S_{i+j-1}', Q_{i+j}) \geq \big(|S_{i+j-1} \setminus
    S_{i+j-1}'| - \eps|S_{i+j-1}|\big) \alpha\tilde sp/4.
  \]
  On the other hand, assuming $S_{i+j-1} \setminus S_{i+j-1}'$ is of order at
  least $\alpha\tilde sp/16$ and, hence,
  \[
    |S_{i+j-1} \setminus S_{i+j-1}'||Q_{i+j}|p \geq \alpha\tilde sp/16 \cdot
    2\eps\tilde sp \geq 100s/\lambda^2
  \]
  for $C > 0$ sufficiently large, property \ref{prp:upper-uniform} implies that
  \[
    e_G(S_{i+j-1} \setminus S_{i+j-1}', Q_{i+j}) \leq (1+\lambda) 2\eps\tilde s
    |S_{i+j-1} \setminus S_{i+j-1}'|p.
  \]
  Since $\eps < \alpha/128$, this is a contradiction. Therefore, there are sets
  $S_{i+j-1}' \subseteq S_{i+j-1}$ of order at least $|S_{i+j-1}| - \alpha\tilde
  sp/16$ for each $j \in [\delta s]$ such that every $v \in S_{i+j-1}'$
  satisfies $|N_G(v, U_{i+j} \setminus Q_{i+j})| \geq \alpha\tilde sp/4$.

  We will now find a collection of sets $S_{i+j-1}'' \subseteq S_{i+j-1}'$ of
  order at least $|S_{i+j-1}| - \alpha\tilde sp/8$ such that, for every $2 \leq
  j \leq \delta s$, every $v \in S_{i+j-2}''$ has a non-empty
  $N_G(v,S_{i+j-1}'')$. First, choose $S_{i+\delta s-1}'' \subseteq S_{i+\delta
  s-1}'$ of order $|S_{i+\delta s-1}| - \alpha\tilde sp/8$ arbitrarily, noting
  that such a set exists by the bound on $|S_{i+\delta s-1}'|$. Having chosen
  $S_{i+j-1}''$ for some $2 \leq j \leq \delta s$, we choose $S_{i+j-2}''$ as
  follows. Recall that $(S_{i+j-2},S_{i+j-1})$ is $(\eps,\alpha
  p)$-lower-regular and, thus, by Lemma~\ref{lem:slicing-lemma} and the bounds
  on the orders of $S_{i+j-2}'$ and $S_{i+j-1}''$, $(S_{i+j-2}',S_{i+j-1}'')$ is
  $(2\eps,\alpha p)$-lower-regular. It follows that there are at least
  $(1-2\eps)|S_{i+j-2}'| \geq |S_{i+j-2}| - \alpha\tilde sp/8$ vertices $v \in
  S_{i+j-2}'$ which satisfy
  \[
    \deg_G(v, S_{i+j-1}'') \geq (1-2\eps)|S_{i+j-1}''|\alpha p \geq \alpha^2
    \tilde sp^2/16 > 0.
  \]
  We declare the set of such vertices to be $S_{i+j-2}''$ and continue on to the
  next index $j$.

  Starting with an arbitrary $v_1 \in S_i''$ and sequentially choosing $v_j \in
  N_G(v_{j-1}, S_{i+j-1}'')$ now completes the proof.
\end{proof}

{\small \bibliographystyle{abbrv} \bibliography{references}}

\end{document}